\documentclass[letterpaper, 10 pt, conference]{ieeeconf}  % Comment this line out if you need a4paper
\IEEEoverridecommandlockouts                              % This command is only needed if 
\let\labelindent\relax
\usepackage{enumitem}
\usepackage{graphicx}
\usepackage{amssymb,mathtools,amsfonts}
\usepackage{gensymb}
  
\usepackage{amsthm}  
\usepackage{caption, subcaption}
\usepackage{hyperref}
\usepackage{cleveref}
\usepackage{multirow}
\usepackage{color}
\usepackage{psfrag}
\usepackage{empheq}
\usepackage{bm} % bold italic symbols in math mode
\usepackage{url}
\usepackage{float}
\usepackage{comment}
\usepackage{alphalph}
\usepackage[short]{optidef}
\usepackage[dvipsnames]{xcolor}
\usepackage{tikz}
\usepackage{tikz-cd}
\usepackage{suffix}

\usetikzlibrary{positioning}
\usetikzlibrary{arrows}
\usetikzlibrary{babel}
\usetikzlibrary{automata}
\usetikzlibrary{shapes, snakes}
\usetikzlibrary{positioning, fit, calc}
\usetikzlibrary{decorations.pathreplacing}
\usetikzlibrary{patterns}
\usetikzlibrary{patterns.meta}
\usetikzlibrary{arrows.meta}
\usepackage{pgfplots}
\pgfplotsset{compat=newest}
\usepackage{pgfplotstable}
\usepackage{pgf}

\usepackage[export]{adjustbox}

\newcommand{\R}{{\mathbb{R}}}
\newcommand{\dd}{{\mathrm{d}}}

\newcommand{\activeset}{{\mathcal{A}}}
\newcommand{\A}{{\mathcal{S}}} % active set
\newcommand{\ns}{{n_\mathrm{s}}}

\newcommand{\Nfe}{{N_\mathrm{fe}}}

\newcommand{\paren}[1]{\left( #1 \right)}
\newcommand{\brac}[1]{\left[ #1 \right]}
\newcommand{\cbrac}[1]{\left\{ #1 \right\}}
\WithSuffix\newcommand\paren*[1]{( #1 )}
\WithSuffix\newcommand\brac*[1]{[ #1 ]}
\WithSuffix\newcommand\cbrac*[1]{\{ #1 \}}

\newcommand{\vo}{\vec{o}\@ifnextchar{^}{\,}{}}

\newcommand{\Set}[2]{\left\{\, #1 \mid #2\,\right\}}
\WithSuffix\newcommand\Set*[2]{\{\, #1 \mid #2\,\}}
\newcommand{\norm}[1]{{\left\Vert#1\right\Vert}}
\newcommand{\transp}[1]{{#1}^\top}

\newcommand{\subspace}{\mathcal{E}}

\newcommand{\C}{{\mathcal{C}}}
\newcommand{\K}{{\mathcal{K}}}
\newcommand{\xdot}{{\dot{x}}}
\newcommand{\proj}[2]{\mathrm{P}_{#1}\paren{#2}}
\WithSuffix\newcommand\proj*[2]{\mathrm{P}_{#1}\paren*{#2}}
\newcommand{\tancone}[2]{\mathcal{T}_{#1}\paren{#2}}
\newcommand{\normcone}[2]{\mathcal{N}_{#1}\paren{#2}}

\newcommand{\nosnoc}{\texttt{nosnoc}}

\DeclareMathOperator*{\argmin}{arg\,min}
\DeclareMathOperator*{\Span}{span}

\newtheorem{theorem}{Theorem}
\newtheorem{example}{Example}

\newtheorem{lemma}[theorem]{Lemma}

\begin{document}

\title{First-Order Sweeping Processes and Extended Projected Dynamical Systems:
Equivalence, Time-Discretization and Numerical Optimal Control}
\author{
  Anton Pozharskiy$^{1}$, Armin Nurkanovi\'c$^{1}$, Moritz Diehl$^{1,2}$
  \thanks{This research was supported by DFG via Research Unit FOR 2401, project 424107692 and 525018088, by BMWK via 03EI4057A and 03EN3054B, and by the EU via ELO-X 953348.}
  \thanks{
  $^1$Department of Microsystems Engineering (IMTEK),
  $^2$Department of Mathematics, University of Freiburg, Germany,
  \sloppy\texttt{\{anton.pozharskiy,armin.nurkanovic,moritz.diehl\} @imtek.uni-freiburg.de}
	}
}
\maketitle
\thispagestyle{empty} % Removes the page number in the first page
% \copyrightnotice
\begin{abstract}
  % There exists a wide class of constrained dynamic system, i.e., systems such that by some means the state stays within a set.
  Constrained dynamical systems are systems such that, by some means, the state stays within a given set.
  Two such systems are the (perturbed) Moreau sweeping process and the recently proposed extended Projected Dynamical System (ePDS)~\cite{Sharif2019}.
  We show that under certain conditions solutions to the ePDS correspond to the solutions of a dynamic complementarity system, similar to the one equivalent to ordinary PDS.
  We then show that the perturbed sweeping process with time varying set can, under similar conditions, be reformulated as an ePDS.
  In this paper, we leverage these equivalences to develop an accurate discretization method for perturbed first-order Moreau sweeping processes via the finite elements with switch detection method~\cite{Nurkanovic2024a}.
  This allows the efficient optimal control of systems governed by ePDS and perturbed first-order sweeping processes.
\end{abstract}
\section{Introduction}
In recent years there has been renewed interest in nonsmooth dynamical systems in general and particularly in Projected Dynamical Systems (PDS) and sweeping processes.
When PDS were first introduced by Dupuis and Nagurney~\cite{Dupuis1993}, they involved smooth ODEs constrained to an set $\C$.
The requirements on the coupled set has since then been relaxed to simply ones where the tangent cone is convex, i.e. the set is uniformly prox-regular.
This class of systems have been used to study various systems including vaccination strategies~\cite{Cojocaru2007} and economic problems~\cite{Nagurney1995}.
There has also been recent work connecting PDS with the popular field of control barrier functions~\cite{Delimpaltadakis2024}.
The resurgence of interest has lead to several extensions to PDS, namely so called ``Extended'' and ``Oblique'' Projected Dynamical Systems (ePDS and oPDS, respectively).
Both extensions modify the projection operator used to transform the dynamics such that the state stays within $\C$, and have been used to treat novel control system~\cite{Sharif2019, Heemels2020}.

First-order Sweeping Processes (FOSwP) were introduced originally by Moreau~\cite{Moreau1977}.
Similarly to PDS, FOSwP have seen recent work in an optimal control context such as in the case of marine surface vehicles~\cite{Cao2021} and crowd motion~\cite{Cao2021a}.
FOSwP have also been applied in soft robotics contexts~\cite{Colombo2021}.
Recently, a numerical method using measure relaxations has been proposed, though only for stationary sets $\C$~\cite{Chhatoi2024}.

We have previously introduced an accurate discretization method for ordinary PDS using the Finite Elements with Switch Detection (FESD) method~\cite{Pozharskiy2024}.
Due to the equivalence of PDS with FOSwP with time invariant sets (sometimes called reflecting boundaries), the systems handled in \cite{Cao2021, Cao2021a} could also accurately be discretized using this method.
However, if the sweeping sets are time dependent then this equivalence does not hold.
This limitation precludes the discretization of the sweeping processes with time varying sets such as the one in~\cite[Section 8.1]{Colombo2016}.
To address this we look to ePDS as we require a similar remedy as in~\cite{Sharif2019} as projection of the gradient in the direction of time should be disallowed.

\paragraph*{Contributions}
This paper proves that under some assumptions on the time dependence of the sweeping set $\C(t)$, FOSwP can be reformulated into an ePDS.
These ePDS then can be shown to be equivalent to a dynamic complementarity system which can be accurately discretized via a slight modification of the FESD discretization for PDS~\cite{Pozharskiy2024}.
We then provide several examples of optimal control problems with this class of system.
These examples and the discretization method are implemented and available in the open-source software package \nosnoc~\cite{Nurkanovic2022b}.

\paragraph*{Notation}
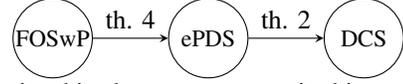
\begin{figure}[t]
  \centering
  % \begin{tikzcd}[arrows={shorten >=-5ex,shorten <=-5ex},nodes={inner sep=0pt}]
  %   \mathrm{FOSwP} \arrow[r, phantom, "\subset"] &\mathrm{ePDS} \arrow[r, phantom, "\subset"] & \mathrm{GCS} \arrow[r, phantom, "\subset"] & \mathrm{DCS}
  % \end{tikzcd}
  \begin{tikzpicture}[>=stealth]
    \node[state,inner sep=0pt,minimum size=3em] (FOSwP) {\small FOSwP};
    \node[state,inner sep=0pt,minimum size=3em] (ePDS) [right=of FOSwP] {\small ePDS};
    \node[state,inner sep=0pt,minimum size=3em] (DCS) [right=of ePDS] {\small DCS};

    \path[->] (ePDS) edge node[above] {th. \ref{th:ePDS-eGCS}} (DCS);
    \path[->] (FOSwP) edge node[above] {th. \ref{theo:mps-epds-eq}} (ePDS);
  \end{tikzpicture}
  \vspace{-0.2cm}
  \caption{Relationships between systems in this paper with edge labels denoting the theorems that relate them.}
  \label{fig:eq}
  \vspace{-0.5cm}
\end{figure}
Regard a closed set $\C\subseteq \R^n$. 
The tangent cone to $\C$ at $x$, is defined as $\tancone{\C}{x} = \Set*{{d\in\R^n}}{{\exists\{x_i\}\in\C,\{t_i\}\in\R_+,{x_i\rightarrow x}, {t_i\rightarrow t}, {d = \lim_{i\rightarrow\infty}\frac{x_i-x}{t_i}}}}$.
The normal cone to $\C$ at $x$ is defined by $\normcone{\C}{x} = \Set*{{v\in\R^n}}{\langle v,d\rangle \le 0, \forall d\in\tancone{C}{x}}$, for  $x\in \C$.
In the interior of $\C$, $\tancone{\C}{x} = \R^n$ and $\normcone{\C}{x} = \cbrac{0}$, and both are the empty set outside $\C$.
A finitely defined $\C\subset\R^n$ is written as $\C = {\Set*{x\in\R^n}{c(x)\ge 0}}$ with $c: \R^n\rightarrow\R^{n_c}$.
In the sequel, we also use $\activeset(x)$ to denote the index set of active constraints: $\activeset(x) = {\Set*{i\in \{1,\dots,n_c\}}{c_i(x) = 0}}$.
We also define $\nabla c_\activeset(x)$ as $\nabla c(x)$ with only the columns corresponding to the indices in $\activeset(x)$.
Let $\K \subseteq \R^n$ be a closed convex set.
The Euclidean projection operator is defined via a convex optimization problem ${\proj{\K}{v} =\argmin_{w \in \K}\frac{1}{2}\norm{w-v}^2_2}$.
Let linear subspace $\subspace$ be the space spanned by the columns of the matrix $E\in\R^{n\times n'}$, with $n'\le n$, i.e., ${\subspace = \Set*{s\in\R^n}{E\transp{E}x = s,\ x\in \R^n}\subseteq\R^n}$.
We denote $\subspace_{\perp} = \R^n/ \subspace$ to be the orthogonal complement of $\subspace$ which is itself a linear subspace.
Projection of a vector $v\in\R^n$ onto the linear subspace $\subspace$ has a closed form solution, namely, $\proj{\subspace}{v} = E\transp{E}v$, if the columns of $E$ form an orthonormal basis of $\subspace$.
In the following we use $0_n\in \R^n$ to define a column vector of $n$ zeros.
\section{Background}
In this section, we discuss some of the mathematical background required for the remainder of the paper.
This first covers the properties of the ePDS and the projection operator it uses, which is followed by an overview of the perturbed first-order Moreau sweeping process.
\subsection{Extended PDS}
\begin{figure}[t]
  \centering
  \begin{subfigure}{0.48\linewidth}
    \vspace{0.2cm}
    \centering
    \begin{tikzpicture}
  \begin{axis}[
    width=1.4\linewidth, axis equal image,
    height=1.4\linewidth,
    xmin=0.6,xmax=1.4,
    ymin=0.6,ymax=1.4,
    % anchor=origin,
    % rotate around={45:(current axis.origin)},
    axis x line=none, axis y line=none,
    font=\tiny,
    >=stealth,
    clip=false
    ]
    % Set
    \draw[dashed] (0.6,0.6) -- (1.4,1.4) node[pos=.85, above left, xshift=7pt] {\normalsize$c(x)$};
    %\draw[->, Mulberry, line width=2pt] (1,1) -- (0.7,1.3) node[pos=1, right, yshift=3pt, xshift=0pt] {\normalsize$\normcone{\C}{x}$};
    \draw[color=OliveGreen, ->, line width=2pt] (1,1) -- (.7,1.15) node[below left] {\normalsize$v$};
    \draw[->, Cerulean, line width=2pt] (1,1) -- (1.3,1.3);
    \draw[->, Cerulean, line width=2pt] (1,1) -- (0.7,0.7) node[pos=1, below right, yshift=2pt] {\normalsize$\tancone{\C}{x}$};
    \draw[draw=none, fill=Cerulean, opacity=0.2, rotate around={-135:(1,1)}] (1,1) -- (1.3,1) arc[start angle=0,end angle=180,radius=0.3] -- (1,1);
    \draw[<->, Orange, line width=2pt] (0.6,1) -- (1.4,1) node[pos=0.9, below right] {\normalsize$\subspace$};
    \draw[->, line width=2pt] (1,1) -- (1.15,1.15) node[pos=0.7, below right, yshift=1pt, xshift=-5pt] {\normalsize$\proj{\tancone{\C}{x},\subspace}{v}$};
    \draw[->, dashed, Orange, line width=2pt] (0.7,1.15) -- (1.15,1.15) node[pos=.65, above, yshift=-2pt] {\normalsize$w-v\in\subspace$};
    %\addplot[->, Red, line width=2pt] (1, 1) -- (0.925,0.925) node[below right] {\normalsize$\proj{\tancone{\C}{x}}{f(x)}$};
    %\draw[->, dashed, Red, line width=2pt] (0.7,1.15) -- (0.925,0.925) node[pos=.85, left, yshift=-2pt] {\normalsize$w-v\in\normcone{\C}{x}$};
  \end{axis}
\end{tikzpicture}
%%% Local Variables:
%%% mode: latex
%%% TeX-master: "../ecc2025_epds_foswp"
%%% End:
    \vspace{-0.7cm}
    \caption{$\proj{\tancone{\C}{x},\subspace}{v}$.}
  \end{subfigure}
  \hfill
  \begin{subfigure}{0.48\linewidth}
    \vspace{0.2cm}
    \centering
    \begin{tikzpicture}
  \begin{axis}[
    width=1.4\linewidth, axis equal image,
    height=1.4\linewidth,
    xmin=0.6,xmax=1.4,
    ymin=0.6,ymax=1.4,
    % anchor=origin,
    % rotate around={45:(current axis.origin)},
    axis x line=none, axis y line=none,
    font=\tiny,
    >=stealth,
    clip=false
    ]
    % Set
    \draw[dashed] (0.6,0.6) -- (1.4,1.4) node[pos=.85, above left, xshift=7pt] {\normalsize$c(x)$};
    %\draw[->, Mulberry, line width=2pt] (1,1) -- (0.7,1.3) node[pos=1, right, yshift=3pt, xshift=0pt] {\normalsize$\normcone{\C}{x}$};
    \draw[color=OliveGreen, ->, line width=2pt] (1,1) -- (.7,1.15) node[below left] {\normalsize$v$};
    \draw[->, Cerulean, line width=2pt] (1,1) -- (1.3,1.3);
    \draw[->, Cerulean, line width=2pt] (1,1) -- (0.7,0.7) node[pos=1, below right, yshift=2pt] {\normalsize$\tancone{\C}{x}$};
    \draw[draw=none, fill=Cerulean, opacity=0.2, rotate around={-135:(1,1)}] (1,1) -- (1.3,1) arc[start angle=0,end angle=180,radius=0.3] -- (1,1);
    %\draw[<->, Orange, line width=2pt] (0.6,1) -- (1.4,1) node[pos=0.9, below right] {\normalsize$\subspace$};
    %\draw[->, line width=2pt] (1,1) -- (1.15,1.15) node[pos=0.7, below right, yshift=1pt, xshift=-5pt] {\normalsize$\proj{\tancone{\C}{x},\subspace}{v}$};
    %\draw[->, dashed, Orange, line width=2pt] (0.7,1.15) -- (1.15,1.15) node[pos=.65, above, yshift=-2pt] {\normalsize$w-v\in\subspace$};
    \addplot[->, Red, line width=2pt] (1, 1) -- (0.925,0.925) node[below right] {\normalsize$\proj{\tancone{\C}{x}}{v}$};
    \draw[->, dashed, Red, line width=2pt] (0.7,1.15) -- (0.925,0.925) node[pos=.85, left, yshift=-2pt] {\normalsize$w-v$};
  \end{axis}
\end{tikzpicture}
%%% Local Variables:
%%% mode: latex
%%% TeX-master: "../ecc2025_epds_foswp"
%%% End:
    \vspace{-0.7cm}
    \caption{$\proj{\tancone{\C}{x}}{v}$.}
  \end{subfigure}
  \caption{Comparing the extended projection operator (left) and the Euclidean projection operator (right).}
  \label{fig:ePDS-proj}
\end{figure}
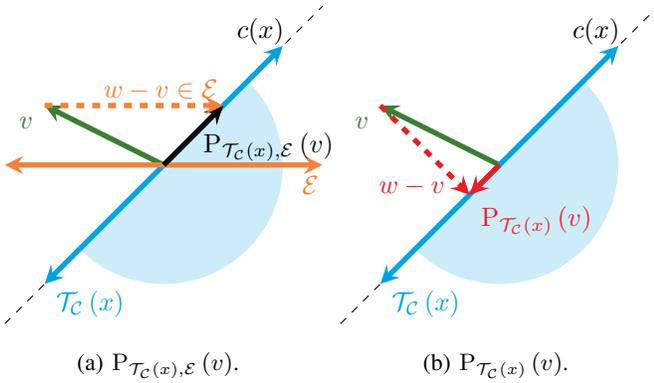
The extended projection operator $\proj{\K,\subspace}{f(x,u)}$ is defined as:
\begin{subequations}
  \label{eq:e-proj}
  \begin{align}
    \proj{\K,\subspace}{v} = &\argmin_{w\in \R^n}\frac{1}{2}\norm{w-v}_2^2,\\
                             & \textrm{s.t.}\quad w\in \K,\ w-v\in\subspace
  \end{align}
\end{subequations}
This convex Quadratic Program (QP) is, intuitively, equivalent to finding the closest point $w\in\C$ to $v$ which can be reached via a straight line in the subspace $\subspace$.
An example of the difference between the the standard Euclidean projection and this operator can be seen in \Cref{fig:ePDS-proj}.
If we apply this projection operator with $\K = \tancone{\C}{x}$ and $v = f(x,u)$, the resulting QP is necessarily convex if the tangent cone $\tancone{\C}{x}$ is convex and its feasible set, $\tancone{C}{x}\cap f(x,u) + \subspace\neq \emptyset$.
This yields the extended PDS, introduced in~\cite{Sharif2019}, which takes the form:
\begin{subequations}\label{eq:ePDS}
  \begin{align}
    \xdot &= \proj{\tancone{\C}{x},\subspace}{f(x,u)},\label{eq:ePDS:xdot}\\
    x(0)&\in \C,
  \end{align}
\end{subequations}
with time-independent ${\C = \Set*{x\in\R^n}{c(x)\ge 0}}$ and $\subspace$ which is a linear subspace of $\R^n$.
In the context of these systems the nonemptiness of QP's feasible set is also useful in conjunction with~\cite[Theorem 2.4]{Cornet1983} to guarantee that the state $x$ stays in $\C$ for all $t\in [0, T]$.
\subsection{First-Order Sweeping Processes}
We treat the perturbed first-order sweeping process:
\begin{align}\label{eq:MSP}
  \xdot&\in f(x,u) -\normcone{\C(t)}{x},
\end{align}
with the time-dependent set ${\C(t) = \Set*{x\in\R^n}{c(x,t)\ge 0}}$, defined by the function ${c:\R^{n_x}\times\R\rightarrow \R^{n_c}}$.
We assume ${f: \R^{n_x}\times\R^{n_u} \rightarrow \R^{n_x}}$ is at least once differentiable in both $x$ and $u$ throughout this paper.

\section{Equivalence of ePDS to a DCS}
In this section, we discuss the equivalence of the ePDS to particular Dynamic Complementarity System (DCS).
The set $\C$ is considered to be time-independent and finitely defined by $c:\R^{n_x}\rightarrow\R^{n_c}$.
\begin{figure}[t]
  \centering
  \vspace{0.2cm}
  \begin{tikzpicture}
  \begin{axis}[
    width=.68\linewidth, axis equal image,
    height=.68\linewidth,
    xmin=0.6,xmax=1.4,
    ymin=0.6,ymax=1.4,
    % anchor=origin,
    % rotate around={45:(current axis.origin)},
    axis x line=none, axis y line=none,
    font=\tiny,
    >=stealth,
    clip=false
    ]
    % Set
    \draw[dashed] (0.6,0.6) -- (1.4,1.4) node[pos=.85, above left, xshift=7pt] {\normalsize$c(x)$};
    %\draw[->, Mulberry, line width=2pt] (1,1) -- (0.7,1.3) node[pos=1, right, yshift=3pt, xshift=0pt] {\normalsize$\normcone{\C}{x}$};
    \draw[color=OliveGreen, ->, line width=2pt] (1,1) -- (.7,1.15) node[below left] {\normalsize$v$};
    \draw[->, Cerulean, line width=2pt] (1,1) -- (1.3,1.3);
    \draw[->, Cerulean, line width=2pt] (1,1) -- (0.7,0.7) node[pos=1, below right, yshift=2pt] {\normalsize$\tancone{\C}{x}$};
    \draw[draw=none, fill=Cerulean, opacity=0.2, rotate around={-135:(1,1)}] (1,1) -- (1.3,1) arc[start angle=0,end angle=180,radius=0.3] -- (1,1);
    \draw[<->, dashed, Orange, line width=1pt] (0.6,1) -- (1.4,1) node[pos=0.1, above left] {\normalsize$\subspace$};
    \draw[->, Orange, line width=2pt] (1,1) -- (.7,1) node[pos=0.9, below] {\normalsize$Ev_\subspace$};
    \draw[<->, dashed, Magenta, line width=1pt] (1,0.6) -- (1,1.4) node[pos=0.9, right] {\normalsize$\subspace_{\perp}$};
    \draw[->, Magenta, line width=2pt] (1,1) -- (1,1.15) node[pos=0.9, left] {\normalsize$E_\perp v_\perp$};
    \draw[->, line width=2pt] (1,1) -- (1.15,1.15) node[pos=0.7, below right, yshift=3pt, xshift=-5pt] {\normalsize$\phi'$};
    \draw[->, Plum, line width=2pt] (1,1) -- (1.15, 1) node[pos=0.7, below] {\normalsize$E\phi_\subspace$};
    %\draw[->, Plum, line width=2pt, opacity=0.5] (1.15,1) -- (1.4, 1) node[pos=0.7, above, opacity=1] {\normalsize$E\phi_\subspace$};
    %\addplot[->, Red, line width=2pt] (1, 1) -- (0.925,0.925) node[below right] {\normalsize$\proj{\tancone{\C}{x}}{f(x)}$};
    %\draw[->, dashed, Red, line width=2pt] (0.7,1.15) -- (0.925,0.925) node[pos=.85, left, yshift=-2pt] {\normalsize$w-v\in\normcone{\C}{x}$};
  \end{axis}
\end{tikzpicture}
%%% Local Variables:
%%% mode: latex
%%% TeX-master: "../ecc2025_epds_foswp"
%%% End:
  \vspace{-0.4cm}
  \caption{Schematic of vectors in \cref{eq:alt_proj} applied to $\K = \tancone{\C}{x}$.}
\end{figure}
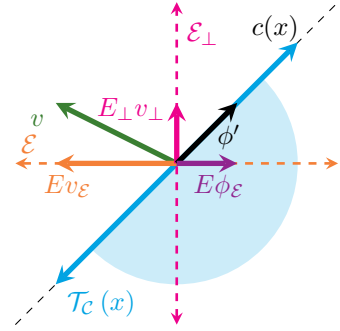
\subsection{Dynamic Complementarity System}

% Differentiating the complementarity conditions with respect to time yields:
% \begin{equation*}
%   0\le \dot{c}=\transp{\nabla c_k}(f(x))+ E\transp{E}{\nabla c_k}\lambda_k)\perp \lambda_k\ge 0.
% \end{equation*}
% if $c(x,t) = 0$ and $f(x)$ points out of the set we get $\transp{\nabla c_k}(f(x))+ E{\nabla c_k}\lambda_k) = 0$.
% Solving for ${\lambda_k = -\inv{(\transp{\nabla c_k}E{\nabla c_k})}\transp{\nabla c_k}f(x)}$ gives us that if $\transp{\nabla c_k}E{\nabla c_k}$ is invertible along a trajectory $x$ then a solution for $\lambda_k$ exists and therefore solutions to this GCS exist.

We now prove that the ePDS can be reformulated into DCS.

%%%%% old proof method
In preparation for that we prove the following lemma for an equivalent projection definition of $\proj{K,\subspace}{v}$:
\begin{lemma}\label{lem:alt_proj}
  The subspace projection operator onto convex cone $K \subset\R^n$ with $\subspace = \Span E \subseteq \R^n$:
  \begin{subequations}
    \label{eq:eproj}
    \begin{align}
      \phi = \proj{\K, \subspace}{v} = &\argmin_{w\in\R^n}\frac{1}{2}\norm{w-v}^2_2,\label{eq:eproj:min}\\
                                       &\textrm{s.t.}\ w\in \K,\ w-v\in\subspace
    \end{align}
  \end{subequations}
  is equivalent to the projection in the subspace (i.e. $\phi = \phi'$):
  \begin{subequations}\label{eq:alt_proj}
    \begin{align}
      v_{\subspace}&= \transp{E}v,\ v_{\perp} = \transp{E_{\perp}}v,\\
      \phi' &= E\phi_{\subspace} + E_{\perp}v_{\perp},\\
      \phi_{\subspace} &= \argmin_{\tilde{w}\in\R^{n'}}\frac{1}{2}\norm{\tilde{w} - v_{\subspace}}^2_2,\label{eq:alt_proj:min}\\
                   &\quad\ \textrm{s.t.}\ E\tilde{w} + E_{\perp}v_{\perp}\in K\label{eq:alt_proj:inc}
    \end{align}
  \end{subequations}
  where the columns of $E\in\R^{n\times n'}$ form orthonormal basis for the subspace $\subspace$ and the columns of $E_{\perp}\in\R^{n\times (n-n')}$ form an orthonormal basis for the orthogonal subspace $\subspace_{\perp}$.
  Without loss of generality, assume that $E$ and $E_{\perp}$ are both selection matrices, i.e. $[E E_{\perp}] = I\in \R^{n\times n}$, and form an orthogonal decomposition of the space $\R^n$.
\end{lemma}
\begin{proof}
  % The assumptions on the structure of $E$, $E_{\perp}$, simplify the equivalent projection to
  % \begin{align*}
  %   v &= (v_{\subspace}, v_{\perp}),\ \phi' = (\phi_{\subspace},v_{\perp}),\\
  %   \phi_{\subspace} &= \argmin_{\tilde{w}\in\R^{n'}}\frac{1}{2}\norm{\tilde{w} - v_{\subspace}}^2_2,\\
  %   &\quad\ s.t.\ (\tilde{w},v_{\perp})\in \K
  % \end{align*}
  % and enforces that $\transp{E_{\perp}} \phi = v_{\perp} = \transp{E_{\perp}} \phi'$ due to the orthogonal decomposition of $\tilde{w}-v$.
  Note that both forms of the projection operator, \cref{eq:eproj} and \cref{eq:alt_proj}, are strictly convex and have corresponding feasible sets in terms of $\phi$ and $\phi'$.
  This comes directly from the fact that ${E\phi_{\subspace} + \transp{E_{\perp}}v_{\perp} - v = (\phi_{\subspace}-v_{\subspace}, 0)\in\subspace}$ and the constraint in \Cref{eq:alt_proj:inc}.
  Now we show that $\phi = \phi'$ for all $\subspace$, $\K$, and $v$.
  Assume for contradiction that $\phi_{\subspace}\neq \transp{E}\phi$.
  If $\norm{\phi - v}^2_2 < \norm{\phi_{\subspace} - v_{\subspace}}^2_2$ then $\norm{\transp{E}\phi - v_{\subspace}}^2_2 < \norm{\phi_{\subspace} - v_{\subspace}}^2_2$ which contradicts \Cref{eq:alt_proj:min}.
  If $\norm{\phi - v}^2_2 > \norm{\phi_{\subspace} - v_{\subspace}}^2_2$ then $\norm{(\phi_{\subspace}, v_{\perp}) - v}^2_2 < \norm{\phi - v}^2_2$ which contradicts \Cref{eq:eproj:min}.
  Therefore, ${\phi = \phi'}$.
\end{proof}
%%%%% old proof method

% The constraints in the projection of \cref{prop:alt_proj} can also be interpreted as $(\tilde{w},0_{n-n'})\in K\cap\subspace_{\perp}$ where $K\cap\subspace_{\perp}$ is a closed convex cone if $K$ is a closed convex cone.
% This comes from the fact that the intersection of closed convex cones is itself a closed convex cone and that any subspace of a finite vector space (as $\subspace$ is) is itself a closed convex cone \todo{cite or prove? probably cite}.
% A consequence of this we can rewrite \cref{eq:ePDS} as:
% \begin{subequations}
%   \begin{align}
%     \xdot &= \proj{\tancone{\C}{x}\cap\subspace_{\perp}}{f(x,u)} + \proj{\subspace_{\perp}}{f(x,u)},\\
%     &= \proj{\tancone{\C}{x}\cap\subspace_{\perp}}{E\transp{E}f(x,u)} + E_{\perp}\transp{E_{\perp}}f(x,u),
        %     \end{align}
        %         \end{subequations}
        %         under the assumption that $\tancone{C}{x}\cap f(x,u) + \subspace \neq \emptyset$.
\begin{theorem}\label{th:ePDS-eGCS}
  Assume that $\C$ is a finitely defined set: ${\C = \Set*{x\in\R^{n_x}}{c(x)\ge 0}}$, and $c_i(x)$ for $i \in \activeset(x)$ satisfy the Linear Independence Constraint Qualification (LICQ) for all $x(t)$, i.e., $\nabla c_i(x)$ are linearly independent for $i \in \activeset(x)$.
  Further, assume that for all time ${\tancone{C}{x}\cap f(x,u) + \subspace\neq \emptyset}$.
  The absolutely continuous solutions $x(t)$ to the ePDS in \Cref{eq:ePDS} with $\subspace = \Span E\subseteq \R^{n_x}$ and $x(0)\in\C$, are equivalent to solutions to the DCS:
  \begin{subequations}
    \label{eq:eDCS}
    \begin{align}
      \xdot &= f(x, u) + E\transp{E}\nabla_xc(x)\lambda,\\
      0 &\le \lambda \perp c(x)\ge 0.
    \end{align}
  \end{subequations}
\end{theorem}
\begin{proof}
  %%%% Old Proof
  We apply the same simplifying assumptions as in \Cref{lem:alt_proj} to define $E\in\R^{n_x\times n'}$ and $E_{\perp}\in\R^{n_x\times (n_x-n')}$.
  Recall that the tangent cone to $\C$ at $x$ with active set $\activeset(x)$ is ${\tancone{\C}{x} = \Set*{v\in\R^{n}}{\transp{\nabla c_i(x)}v \ge 0, i \in\activeset(x)}}$.
  This cone is closed and convex if the constraint functions $c_i(x),\ i\in\activeset(x)$ satisfy a constraint qualification such as LICQ~\cite{Rockafellar1997}.
  Now we use the equivalent optimization problem in \Cref{eq:alt_proj} for the right hand side of \Cref{eq:ePDS}, i.e:
  \begin{align*}
    \xdot &= E\xdot_{\subspace} + \transp{E_{\perp}}v_{\perp},\\
    v_{\subspace}&= \transp{E}f(x,u),\ v_{\perp} = \transp{E_{\perp}}f(x,u),\\
    \xdot_{\subspace} &= \argmin_{\tilde{w}\in\R^{n'}}\norm{\tilde{w} - v_{\subspace}},\\    
          &\quad\ \mathrm{s.t.}\ E\tilde{w} + E_{\perp}v_{\perp}\in \tancone{\C}{x}
  \end{align*}
  The constraint set $\tancone{\C}{x}$ in the convex optimization problem above is the polyhedral set:
  \begin{equation*}
    \Set*{\tilde{w}\in \R^{n'}}{\transp{\nabla c_i}(E\tilde{w} + E_{\perp}v_{\perp})\ge 0,\ i\in\activeset(x)},
  \end{equation*}
  which is a convex polytope, which means the whole problem is a convex quadratic program.
  Therefore, we have the Karush-Kuhn-Tucker (KKT) conditions:
  \begin{align*}
    \tilde{w}-v_{\subspace} -\transp{E}\sum_{j\in\activeset(x)}\lambda_j\nabla c_j(x) = 0,\\
    0\le \lambda_j\perp\transp{\nabla c_j}(E\tilde{w} + E_{\perp}v_{\perp})\ge 0,
  \end{align*}
  for $j\in\activeset(x)$ which are both necessary and sufficient.
  Solving the first equation for $\tilde{w}$ and substituting it into the expression for $\xdot$ yields:
  \begin{align*}
    %E\paren{v_{\subspace} + \sum_{j\in\activeset'(w)}\lambda_j\transp{\nabla_xc_j}E} + \transp{E_{\perp}}v_{\perp}\\
    \xdot &= f(x,u) + E\transp{E}\sum_{j\in\activeset(w)}\lambda_j\nabla c_j(x),
  \end{align*}
  which in combination with the complementarity conditions (into which we also substitute the expression for $\tilde{w}$) this can be rewritten as:
  \begin{align*}
    \xdot &= f(x,u) + E\transp{E}\nabla c(x)\lambda\\
    0&\le \lambda_j\perp\transp{\nabla c_j}\xdot\ge 0,\\
    \lambda_l &= 0, \quad l\in \{1,\ldots,n_c\}\setminus\activeset(x),
  \end{align*}
  for $j\in\activeset(x)$.
  Note that $\transp{\nabla c_j}\xdot = \dot{c}(x)$.
  The final condition comes from the definition of $\tancone{\C}{x}$ and is equivalent to $0\le \lambda \perp c(x)\ge 0$.
  % Finally, note that the right hand side of the complementarity is the expression for $\dot{c}_j(x)$.
  % Taking the derivative of the complementarity $0\le \lambda \perp c(x)\ge 0$, with respect to time yields the equivalent expression, i.e. $0\le \lambda \perp c(x)\ge 0$ implies $0\le \lambda_j\perp\transp{\nabla c_j}(E\tilde{w} + E_{\perp}\transp{E_{\perp}}f(x,u))\ge 0$, $j\in\Set*{k}{c_k(x)=0}$.
  This is therefore equivalent to the DCS in \Cref{eq:eDCS}.
  %%%% Old Proof
\end{proof}

With this we have a computationally useful equivalent DCS to which we can apply the previously developed FESD method.

% \begin{remark}
%   We also note that the well-posedness and existence of solutions of \Cref{eq:eDCS} seems to rely either on $\lambda = 0$ or on the invertibility of a certain matrix, namely $A(x) = \transp{\nabla c_\activeset(x)}E\transp{E}\nabla c_\activeset(x)$ where $\nabla c_\activeset(x)$ contains the columns $\nabla c_k(x)$ with ${k\in\activeset(x) = \Set*{i\in \cbrac*{1,\ldots,n_c}}{c_i(x)=0}}$
%   This comes from differentiating \Cref{eq:eDCS} with respect to time once yielding a closed form solution for $\lambda$ in the form:
%   \begin{equation*}
%     \lambda_\activeset =
%     \begin{cases}
%       \inv{A(x)}\nabla c_\activeset(x) f(x,u), & \nabla c_\activeset(x) f(x,u) < 0,\\
%       0, & \nabla c_\activeset(x) f(x,u) \ge 0,
%     \end{cases}
%   \end{equation*}
%   which comes from
%   $$\dot{c}_\activeset(x) = \transp{\nabla c_\activeset(x)}(f(x,u) + E\transp{E}\nabla c_\activeset(x)\lambda) = 0,$$
%   as in the case of the equivalent GCS for a standard PDS.
%   The invertibility of $A(x)$ is not guaranteed in general, and requires the matrix $\transp{E}\nabla c_\activeset(x)$ to have full column rank.
%   However, the assumption that $\tancone{\C}{x} \cap f(x,u) + \subspace$ is nonempty implies that there is some $\nabla c_\activeset\lambda$ such that $\dot{c}_{\activeset} \ge 0$, i.e. for some $\activeset'\subseteq \activeset$, $\transp{E}\nabla c_{\activeset'}(x)$ has full column rank.
% \end{remark}

\section{Time Varying FOSwP as ePDS}
In this section we discuss the reformulation of FOSwP as an ePDS in order to apply the results described in the previous section.
First, a general reformulation of the FOSwP to an ePDS is presented and the equivalence is verified.
This section concludes with an example reformulation of a simple FOSwP to the ePDS.

In the following, the set $\C$ is considered to be time-dependent and finitely defined by $c:\R^{n_x}\times\R\rightarrow\R^{n_c}$.

\subsection{FOSwP reformulation}
We now return to the First-Order Sweeping Process (FOSwP) $\xdot\in f(x,u) -\normcone{\C(t)}{x}$ and the definition of the finitely-defined time dependent set ${\C(t) = \Set*{x\in\R^n}{c(x,t)\ge 0}}$ with ${c:\R^{n_x}\times\R\rightarrow \R^{n_c}}$.
The mapping from $t\in [0,T]$ to $\C(t)$ is assumed to be forward Lipschitz~\cite[Definition 2]{Hauswirth2018} as $c(x,t)$ is assumed to be at least once continuously differentiable in $x$ and Lipschitz continuous in $t$.
The normal cone to this set at time $t$ is the finitely defined cone $\normcone{\C(t)}{x} = \Set*{{d\in\R^{n_x}}}{{d = \sum_{i\in \activeset(x)}\lambda_i\nabla c_{i}(x,t),\ \lambda_i \ge 0}}$~\cite{Acary2008}.
We reformulate the FOSwP into an ePDS by introducing an auxiliary clock state $\tau$, the dynamics of which we do not project.
This allows us to accurately discretize the system using FESD.

The equivalence of an ePDS and the original sweeping process is formalized in \Cref{theo:mps-epds-eq}.
To aid in the proving of this theorem we require the following lemma:
\begin{lemma}\label{lem:msp-dcs}
  For $\C(t)$ uniformly prox-regular and of bounded variation~\cite{Edmond2006} and $f(x,u)$ at least once differentiable there exists a unique continuous solution to the sweeping process \Cref{eq:MSP} which corresponds to the solution to the DCS:
  \begin{subequations}\label{eq:msp-dcs}
    \begin{align}
      \xdot &= f(x,u) + \nabla_xc(x,t)\lambda\\
      0&\le \lambda\perp c(x,t) \ge 0\label{eq:msp-dcs:comp}
    \end{align}
  \end{subequations}
\end{lemma}
The proof of this lemma follows from the transformation in~\cite[Section 4]{Brogliato2010}.
\begin{theorem}\label{theo:mps-epds-eq}
  The solutions of \Cref{eq:MSP} with $\C(t)$ uniformly prox-regular and of bounded variation~\cite{Edmond2006}, and $f(x,u)$ continuously differentiable, correspond to the solutions of the ePDS:
  \begin{subequations}\label{eq:msp-epds}
    \begin{align}
      \dot{y} &= \proj{\tancone{\C}{y},\subspace}{(f(x,u), 1)},\\
      (x(0),0) &\in \C,
    \end{align}
  \end{subequations}
  with ${y = (x,\tau)\in\R^{n_x+1}}$, functions ${\hat{c}(y) = c(x,\tau)}$, ${\hat{c}:R^{n_x+1}\rightarrow\R^{n_c}}$, ${\C = \Set*{y\in\R^{n_x+1}}{\hat{c}(y) \ge 0}}$, and subspace $\subspace = \Set*{s\in\R^n}{E\transp{E}x = s,\ x\in \R^n}$, with ${E = \paren{e_1,\ldots,e_n}\in\R^{(n_x+1)\times n_x}}$ where $e_i$ is the $i$th basis vector of $\R^{n_x+1}$.
\end{theorem}
\begin{proof}
  As a consequence of \Cref{lem:msp-dcs} we have an equivalent DCS for \Cref{eq:MSP}.
  By \Cref{th:ePDS-eGCS}, for $y(0) = (x(0),0)\in \C$ the solution to \Cref{eq:msp-epds} is the solution to the DCS:
  \begin{align*}
    \dot{y} &= (f(x,u),1) + E\transp{E}\nabla_{y}\hat{c}(y)\lambda'\\
    0&\le \lambda'\perp \hat{c}(y) \ge 0
  \end{align*}
  From the structure of $\subspace$, $\dot\tau = 1$ and therefore $\tau(t) = t$.
  The fact that $E\transp{E}\nabla_{y} \hat{c}(y) = \nabla_xc(x,t)$ given $\tau(t) = t$ for a given $x$ comes directly from the fact that the last row and column in $E\transp{E}$ is zero.
  This means that the solution to \Cref{eq:msp-epds} can, by substituting the above two equivalences, be decomposed:
  \begin{align*}
    \xdot &= f(x,t) - \nabla_xc(x,t)\lambda',\\
    \dot{\tau} &= 1\\
    0&\le \lambda'\perp c(x,t) \ge 0
  \end{align*}
  which is  \Cref{eq:msp-dcs} with an autonomous ODE for the numerical time $\tau$.
\end{proof}
\subsection{FOSwP reformulation example}
\begin{figure}[t]
  \centering
  \vspace{0.1cm}
    \centering
    \begin{tikzpicture}
  \begin{axis}[
    width=0.7\linewidth, axis equal image,
    height=0.7\linewidth,
    xmin=-0.1,xmax=3.1,
    ymin=-0.1,ymax=3.1,
    axis x line=center, axis y line=center,
    xlabel={$t$},
    xlabel style={below right},
    ylabel={$x$},
    ylabel style={above left},
    axis line style={shorten >=-3pt},
    view={0}{90},
    >=stealth
    ]
    \draw[-,color=Mahogany, ultra thick] (0, 2) --  (1, 1) -- (3.1, 3.1);
    \addplot[only marks, mark=*, thick, color=Mahogany] ({0},{2)}) node[above right, yshift=7pt] {$y(0)$};
    \addplot[only marks, mark=*, thick, color=Mahogany] ({1},{1)});
    %\node[color=Mahogany,] at (1.65, -0.5) {$x(t)$};
    \draw[dashed,color = gray, thick] (-1,-1) -- (4, 4) node[pos=0.33, below, xshift=10pt] {$c(y)\mathord{=}0$};
    \draw[Orange,very thick,<->] (2.6,0.1) -- (2.6,1.6) node[pos=0.9, right] {$\subspace$};
    \addplot3[domain=0.1:2.9, y domain=0.1:2.9,color=OliveGreen, samples = 5, samples y=5,quiver={u=1, v=-1, scale arrows=0.2, every arrow/.append style={-{Latex[length=1pt 1,width=5pt 1]}}},line width=2pt]{0};
    %\addplot[domain=0:2,color=OliveGreen,-stealth,samples=7,quiver={u=y, v=0, scale arrows=0.15}, thick]{-1};
    %\addplot[domain=0.33:2,color=Plum,-stealth,samples=6,quiver={u=y, v=-x, scale arrows=0.15}, thick]{-1};
    %\addplot[domain=-1:-.33,color=OliveGreen,-stealth,samples=3,quiver={u=y, v=-x, scale arrows=0.15},thick]{-1};
    %\fill[pattern={Lines[angle=-45,distance=5.0]}, pattern color=gray] (-1.2,-1) rectangle (2.1,-2);
  \end{axis}
\end{tikzpicture}
%%% Local Variables:
%%% mode: latex
%%% TeX-master: "../ecc2025_epds_foswp"
%%% End:
    \vspace{-0.7cm}
    \caption{\Cref{ex:simple-swp} ePDS trajectory.}
  \vspace{-0.7cm}
  \label{fig:ex1-sol}
\end{figure}
We now provide a representative tutorial example for this reformulation and show that the solution for the ePDS reformulation is indeed the solution to the FOSwP.
\begin{example}\label{ex:simple-swp}
  Consider a sweeping process as in \Cref{eq:MSP} with ${x\in\R}$, $\C(t) = \Set*{x\in\R}{c(x,t) = x-t\ge 0}$, ${x(0) = 2\in \C(0)}$, and $f(x,t) = -1$.
  The solution to this sweeping process is the piecewise linear function:
  \begin{equation*}
    x(t) = \begin{cases}
      2-t, & 0\le t < 1,\\
      t, & 1 \le t,
    \end{cases}
  \end{equation*}
  as the state travels in the negative direction at a constant rate until it meets the boundary of $\C$ at $t=1$ and then is swept by this boundary in the positive direction at a constant rate.
  We now reformulate this into an ePDS with ${y = (x,\tau)\in\R^2}$, ${c(y) = x-\tau}$, ${f(y,t) = (-1, 1)}$, and $E = (1,0)$.
  This yields the projection matrix:
  \begin{equation*}
    E\transp{E} =
    \begin{bmatrix}
      1 & 0\\
      0 & 0\\
    \end{bmatrix}.
  \end{equation*}
  The normal cone of $\C$ in this case is defined as:
  \begin{equation*}
    \normcone{\C}{y} =
    \begin{cases}
      \emptyset, & y\notin\C,\\
      \{0\}, & c(y) > 0,\\
      \Set*{d\in\R^2}{d = (1,-1)\lambda,\ \lambda \ge 0}, & c(y) = 0,
    \end{cases}
  \end{equation*}
  and the tangent cone:
  \begin{equation*}
    \tancone{\C}{y} =
    \begin{cases}
      \emptyset, & y\notin\C,\\
      \R^2, & c(y) > 0,\\
      \Set*{(a,b)\in\R^2}{a\ge b}, & c(y) = 0.
    \end{cases}
  \end{equation*}
  The solution to this system is again piecewise linear with a switch at $t = \tau(t) = 1$.
  For the first section of the trajectory the tangent cone is $\R^2$ therefore $y = (2-t,t)$ while $c(y) = 2-t - t > 0$, i.e. $t<1$.
  At $t = 1$ the tangent cone becomes $\Set*{(a,b)\in\R^2}{a\ge b}$ which is a half-plane.
  Note that if we use the Euclidean projection operator, the resulting derivative $\dot{y}$ would be $(0,0)$.
  However, as we only allow projection along the $x$ direction the closest point in $\tancone{\C}{x}$ along $\subspace$ is the vector $(1,1)$.
  This vector is exactly tangential to the line $x-\tau = 0$ and for all remaining time $t \ge 1$: $c(y) = 0$ and $\dot{y} = (1,1)$.
  From this we get that the evolution of the first element of $y$ matches exactly the evolution of the state of the FOSwP.
  A graphical representation the trajectories can be seen in \Cref{fig:ex1-sol}.
\end{example}

\section{Finite Elements with Switch Detection}
Direct optimal control functions on a ``first discretize, then optimize'' paradigm.
As such to apply this technique we must discretize the nonsmooth DCS in \Cref{eq:eDCS}.
However, applying standard Runge-Kutta (RK) schemes to the system yields low accuracy, related to the order of the discontinuity~\cite{Stewart2010,Nurkanovic2020}.
In this case, the system is either first order or second order discontinuous~\cite{Calvo2003}, i.e. either the first or second derivative of the state $x$ is discontinuous, which corresponds to entering the boundary of $\C$ and leaving it, respectively.
This yields, at best, $O(h^2)$ error where $h$ is the step size.
To improve on this accuracy we extend the recently introduced FESD for PDS method to this class of extended PDS~\cite{Pozharskiy2024}.

We consider the trajectory entering or exiting the boundary of $\C$ as a ``switch''.
This method can be directly applied to the DCS in \Cref{eq:eDCS} as it demonstrates the same switching behaviors described in~\cite[Proposition 2]{Pozharskiy2024}.
Therefore, the FESD discretized system with $\Nfe$ integration steps, called finite elements, using an RK discretization scheme with $\ns$ stages with Butcher tableau~\cite{Hairer1996}, $a_{i,j}$, $b_j$, $c_i$ for ${i,j = 1,\ldots,\ns}$ is written as:
\begin{subequations}
  \label{eq:discrete-dcs}
  \begin{align}
    &x_{n,i}\! = \! x_{n,0} + h_n\sum_{j=1}^{\ns}a_{i,j}v_{i,j},\\
    &v_{i,j} = f(x_{n,j},u)+ E\transp{E}{\nabla c(x_{n,j})}\lambda_{n,j}\\
    &0 \le c(x_{n,i}) \perp \lambda_{n,i} \ge 0,\\
    &0 \le c(x_{n,i}) \perp \lambda_{n,j} \ge 0,\\
    &0 \le c(x_{n-1,\ns}) \perp \lambda_{n,j} \ge 0,\\
    &0 \le c(x_{n,i}) \perp \lambda_{n-1,\ns} \ge 0,
  \end{align}
\end{subequations}
with $n = 1,\ldots,\Nfe$, $i = 1,\ldots,\ns$, and $j = 1,\ldots,\ns$, and a fixed control $u$.
Note that in this case $h_n$ is allowed to vary which allows the complementarity conditions to be satisfied.
The complementarity constraints here enforce that the active set $\A(x) = \Set*{i\in\cbrac{1,\ldots,n_c}}{c_i(x) = 0}$ stays constant over each finite element.
This is augmented with the step-equilibration constraints in~\cite[III.A]{Pozharskiy2024} which enforce that $h_n\neq h_{n-1}$ only if there is a switch on the boundary between finite elements $n$ and $n-1$.
This discretization accurately identifies switches in the PDS active set and recovers the higher order accuracy of the RK discretization scheme.
Using a higher accuracy discretization allows for a coarser control grid, while maintaining good constraint satisfaction and accuracy.
This discretization is essentially the same as the one found in~\cite{Pozharskiy2024} with the $E\transp{E}\nabla c(x)\lambda$ term replacing $\nabla c(x)\lambda$, and more details can be found there.  
\section{Numerical Experiments}
In this section we first describe an optimal control problem, with which we demonstrate the numerical efficiency of the FESD discretization applied to a first-order sweeping process. 
We then provide an example optimal panning problem with dynamics defined by perturbed sweeping processes with time dependent sets.
These examples are implemented in \texttt{nosnoc}~\cite{Nurkanovic2022b} and solved with the Scholtes' relaxation~\cite{Scholtes2001} and homotopy loop using IPOPT~\cite{Wachter2002} to a complementarity tolerance of $10^{-10}$~\cite{Nurkanovic2024b}.

\subsection{Wave-rider Problem}
\begin{figure}[t]
  \centering
  \vspace{0.2cm}
  \begin{subfigure}[t]{0.49\linewidth}
    \includegraphics[width=\linewidth, valign=t]{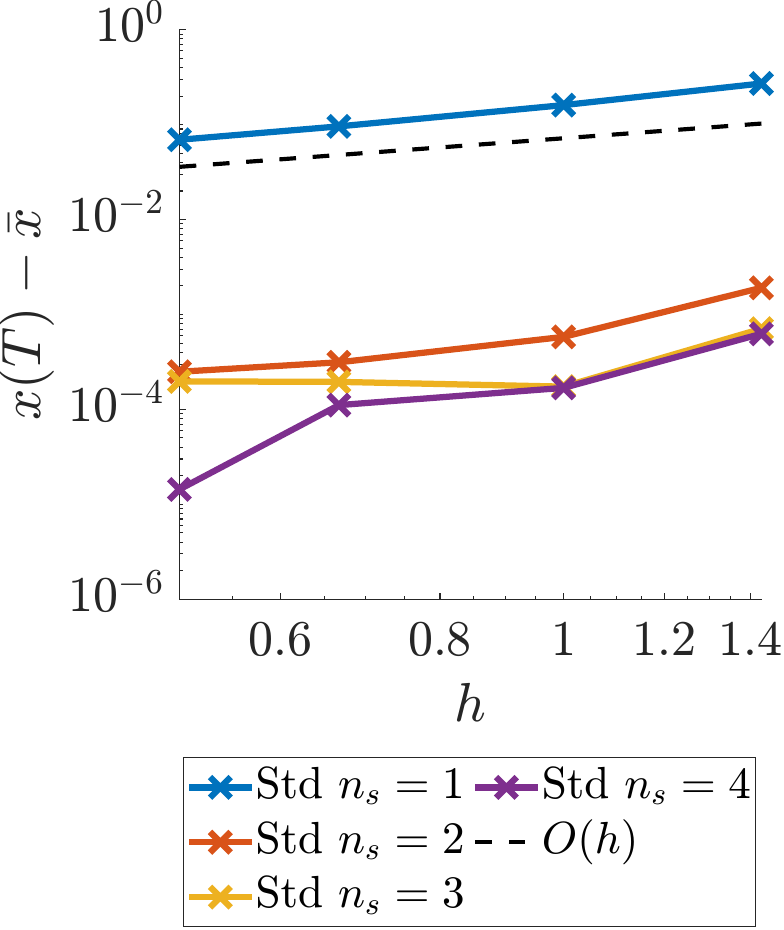}
    \vspace{0.3cm}
    \caption{Fixed time stepping.}
  \end{subfigure}
  \begin{subfigure}[t]{0.49\linewidth}
    %\vspace{-0.01cm}
    \includegraphics[width=\linewidth, valign=t]{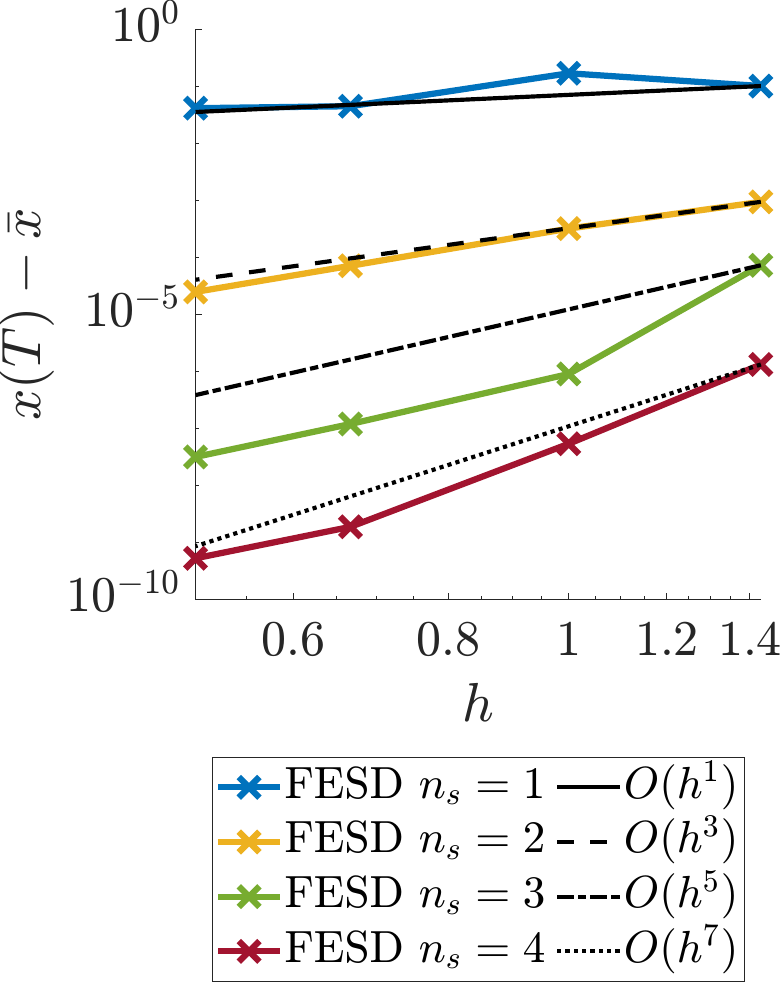}
    \caption{FESD.}
  \end{subfigure}
  \vspace{-0.1cm}
  \caption{Plots of terminal error vs step size.}
  \vspace{-.5cm}
  \label{fig:order-plots}
\end{figure}
This example is a simple optimal control problem (OCP) over a perturbed sweeping process.
The continuous-time OCP is formulated as:
\begin{mini*}
  {\substack{x(\cdot). u(\cdot)}}
  {\int_{0}^{T}\!u(t)^2 \dd t}
  {}
  {}
  \addConstraint{x(t)}{\in (0,u(t)) - \normcone{\C(t)}{x},\ }{t\in[0,T]}
  \addConstraint{x(0)}{=\hat{x}}
  \addConstraint{x(T)}{=\bar{x}}
  \addConstraint{-1\le u(t)}{\le 1,}
\end{mini*}
with $x\in\R^2$, $u\in\R$, $\hat{x} = (0,2)$, $\bar{x} = (1,4)$, and time-varying set ${\C(t) = \Set*{z\in\R^2}{z_1+\cos(z_2-t)\ge 0}}$.
The system is under-actuated and the controller must use the moving boundary defined by $\C(t)$ to translate in the unactuated direction.
We run an experiment with ${N=[5,10,25,50,100]}$ control stages using both the FESD discretization and the standard discretization (${n_s = [1,2,3,4]}$ stage RADAU-IIA integration steps) to show the efficiency of the FESD method applied to such problems.
\Cref{fig:order-plots} shows that even with few control stages the problem discretized with FESD yields a significantly lower terminal error ($\norm{x(T)-\bar{x}}_2$) when accurately simulated using the calculated control.
The control and state trajectory of this OCP can be seen in \Cref{fig:waverider}.
\begin{figure}[t]
  \centering
  \vspace{0.4cm}
  \includegraphics[width=0.9\linewidth]{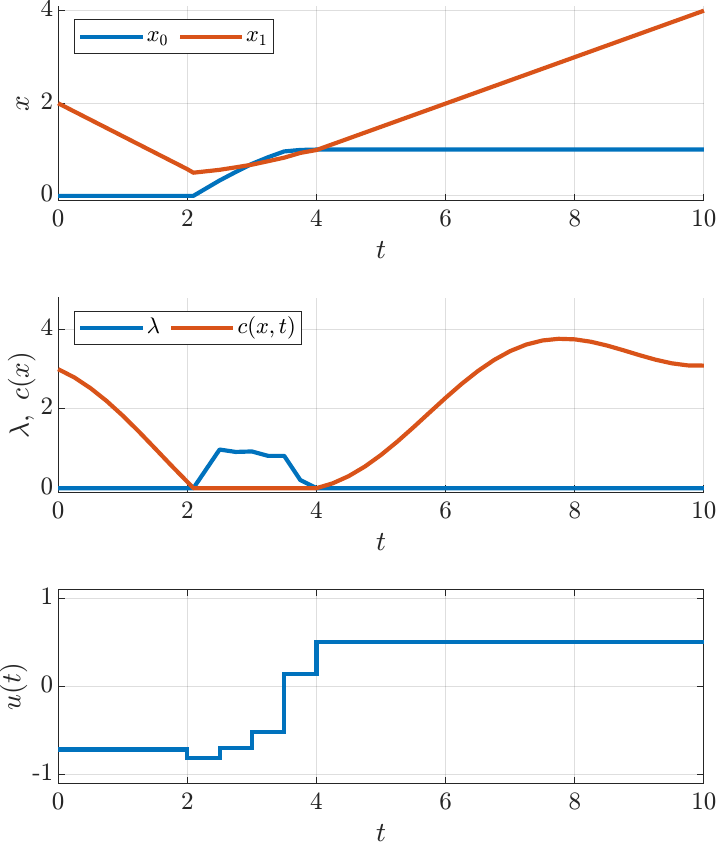}
  \caption{Plot of wave-rider example trajectory.}
  \vspace{-.7cm}
  \label{fig:waverider}
\end{figure}
\subsection{Path Planning with Moving Obstacles}
\begin{figure*}[t]
  \centering
  \vspace{0.2cm}
  \begin{subfigure}{0.19\linewidth}
    \includegraphics[width=\linewidth]{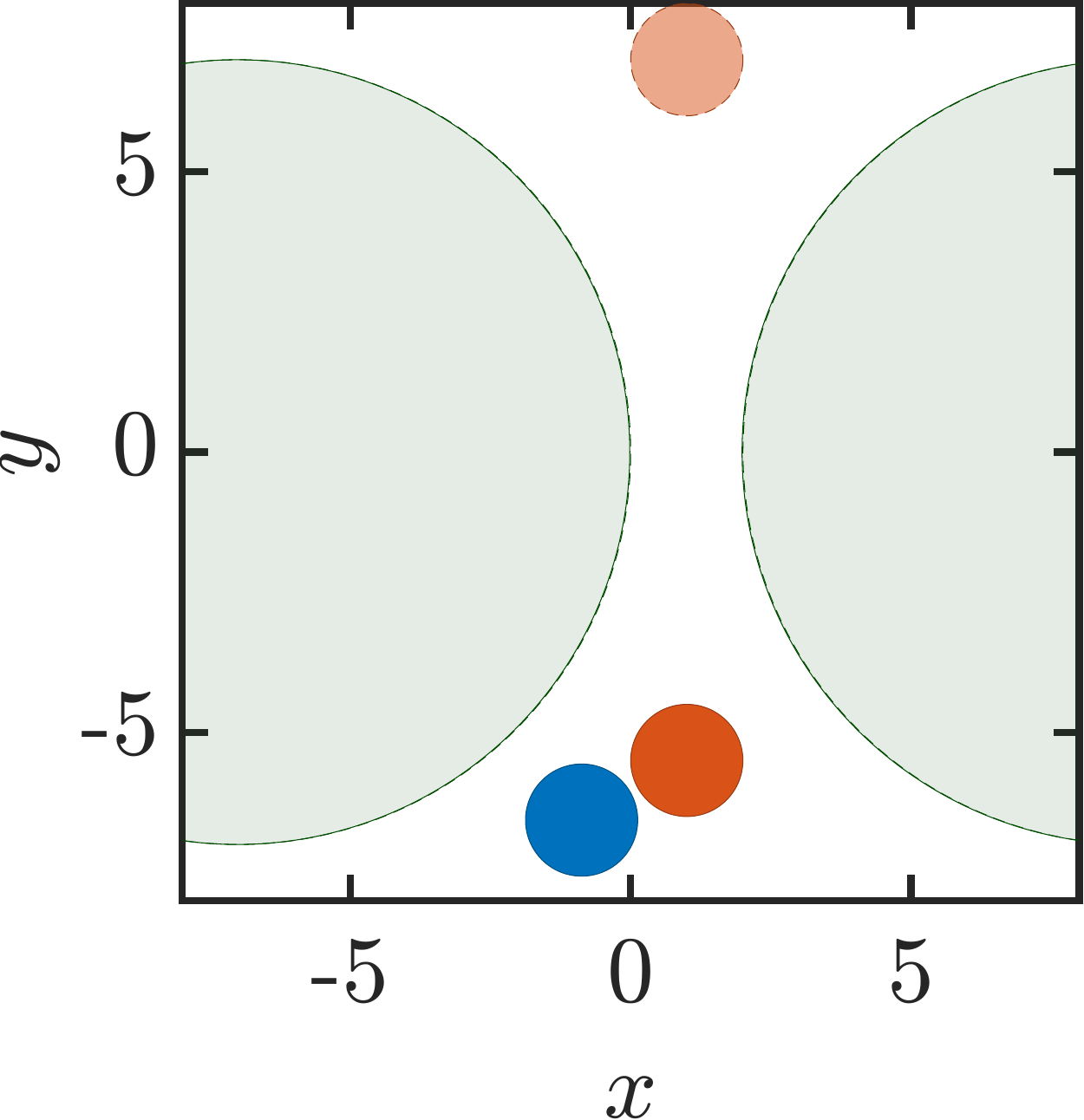}
    \caption{$t=0$}
  \end{subfigure}
  \begin{subfigure}{0.19\linewidth}
    \includegraphics[width=\linewidth]{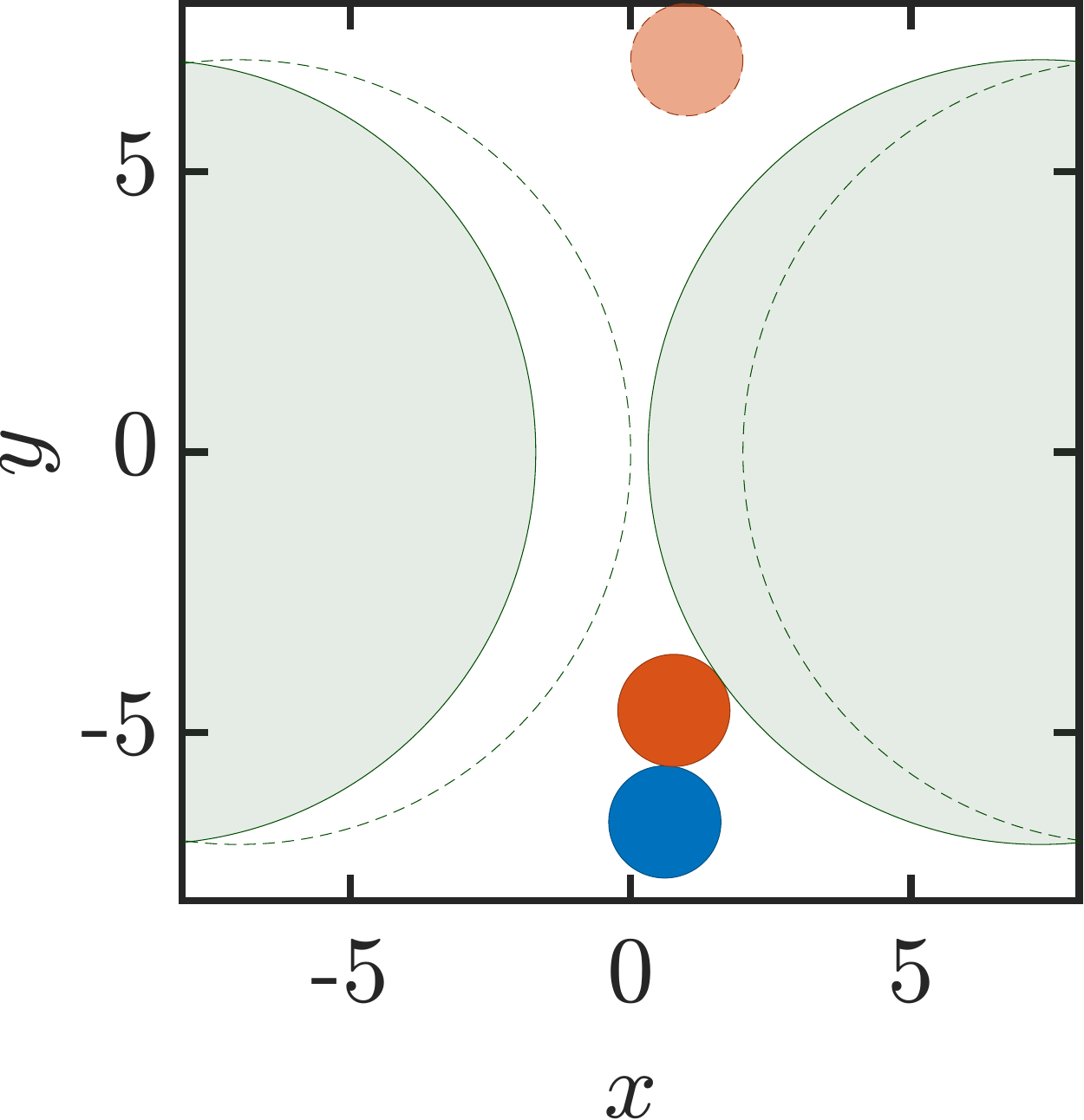}
    \caption{$t=2.7001$}
  \end{subfigure}
  \begin{subfigure}{0.19\linewidth}
    \includegraphics[width=\linewidth]{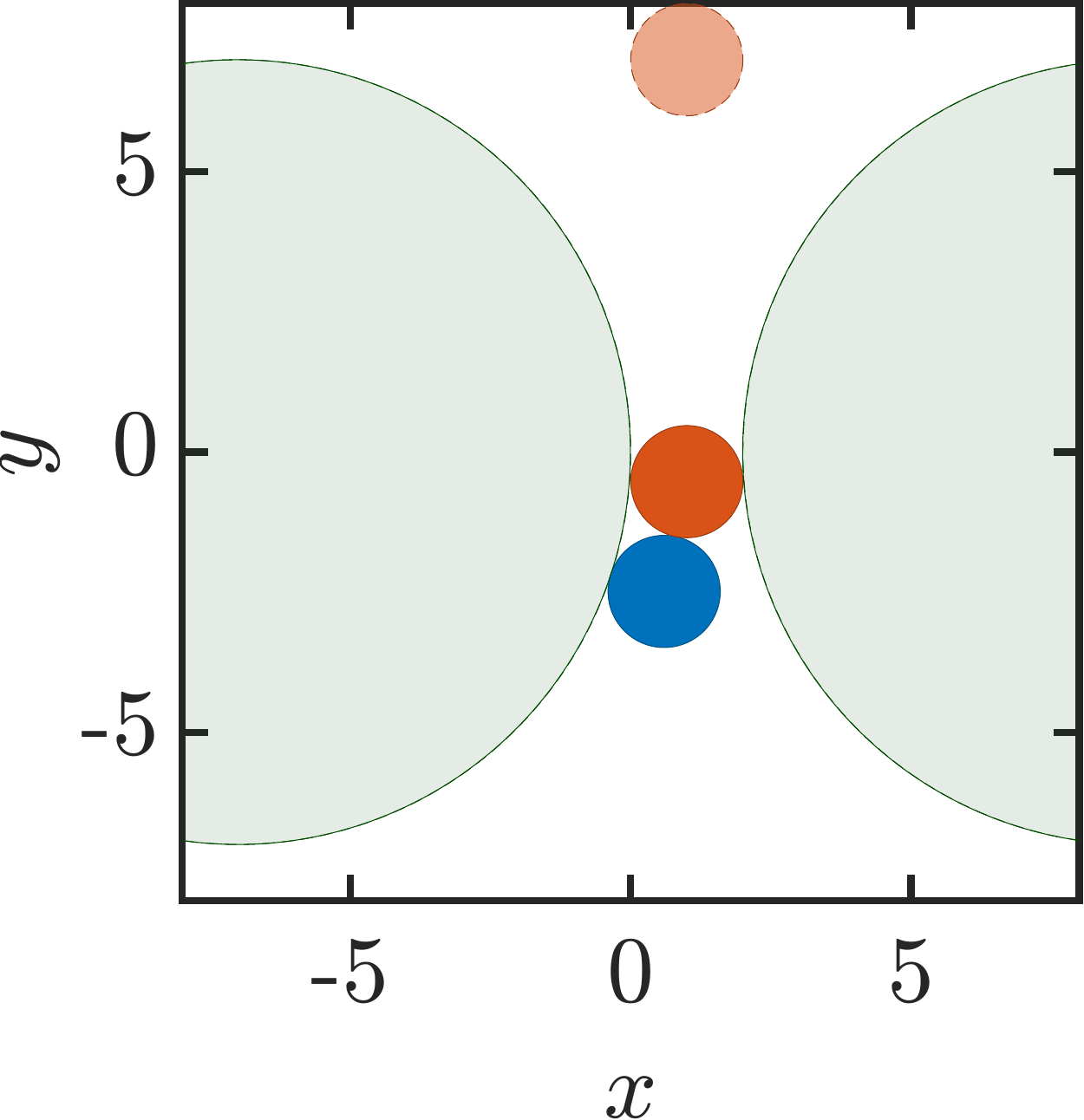}
    \caption{$t=6.3913$}
  \end{subfigure}
  \begin{subfigure}{0.19\linewidth}
    \includegraphics[width=\linewidth]{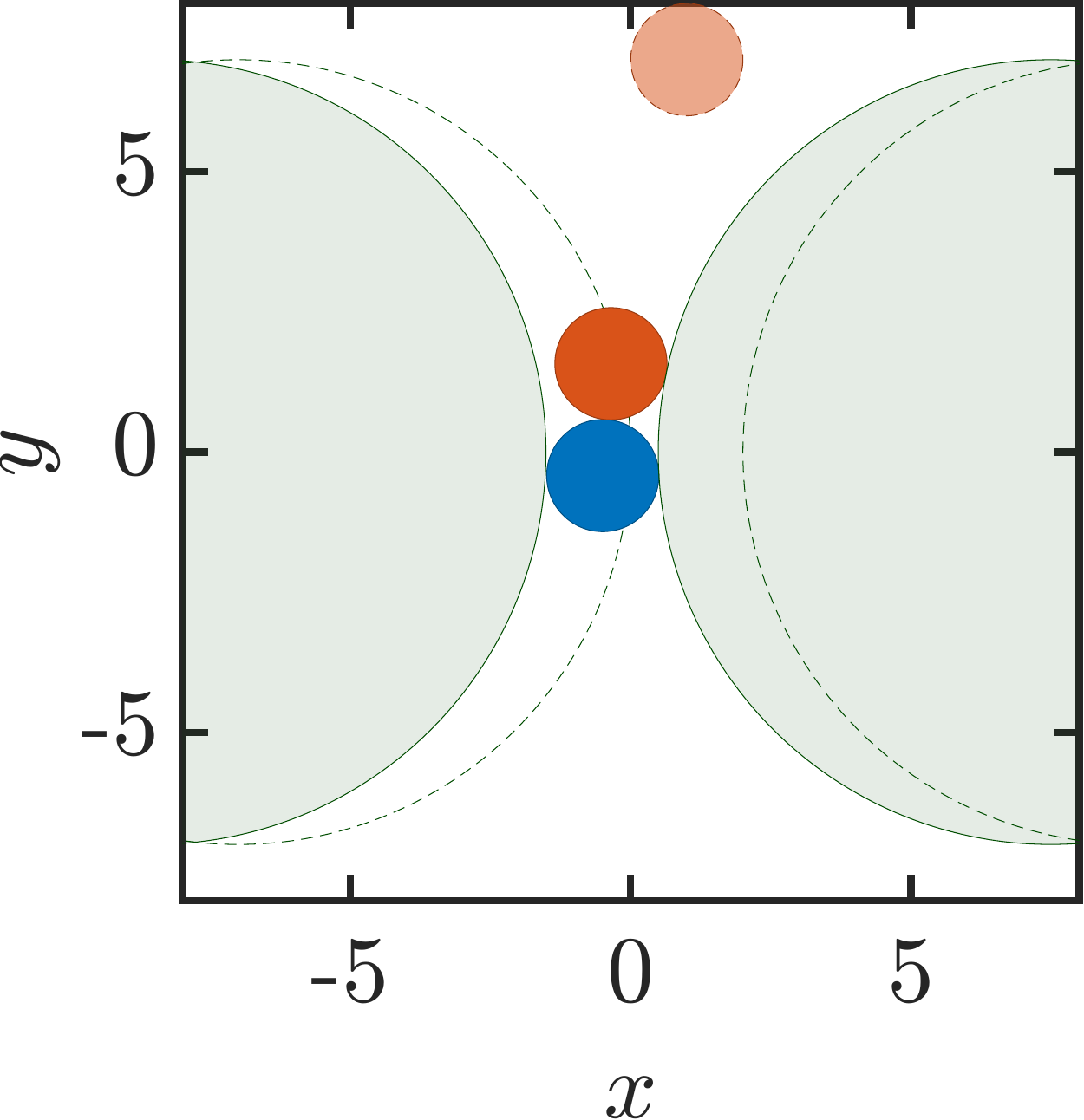}
    \caption{$t=9.0000$}
  \end{subfigure}
  \begin{subfigure}{0.19\linewidth}
    \includegraphics[width=\linewidth]{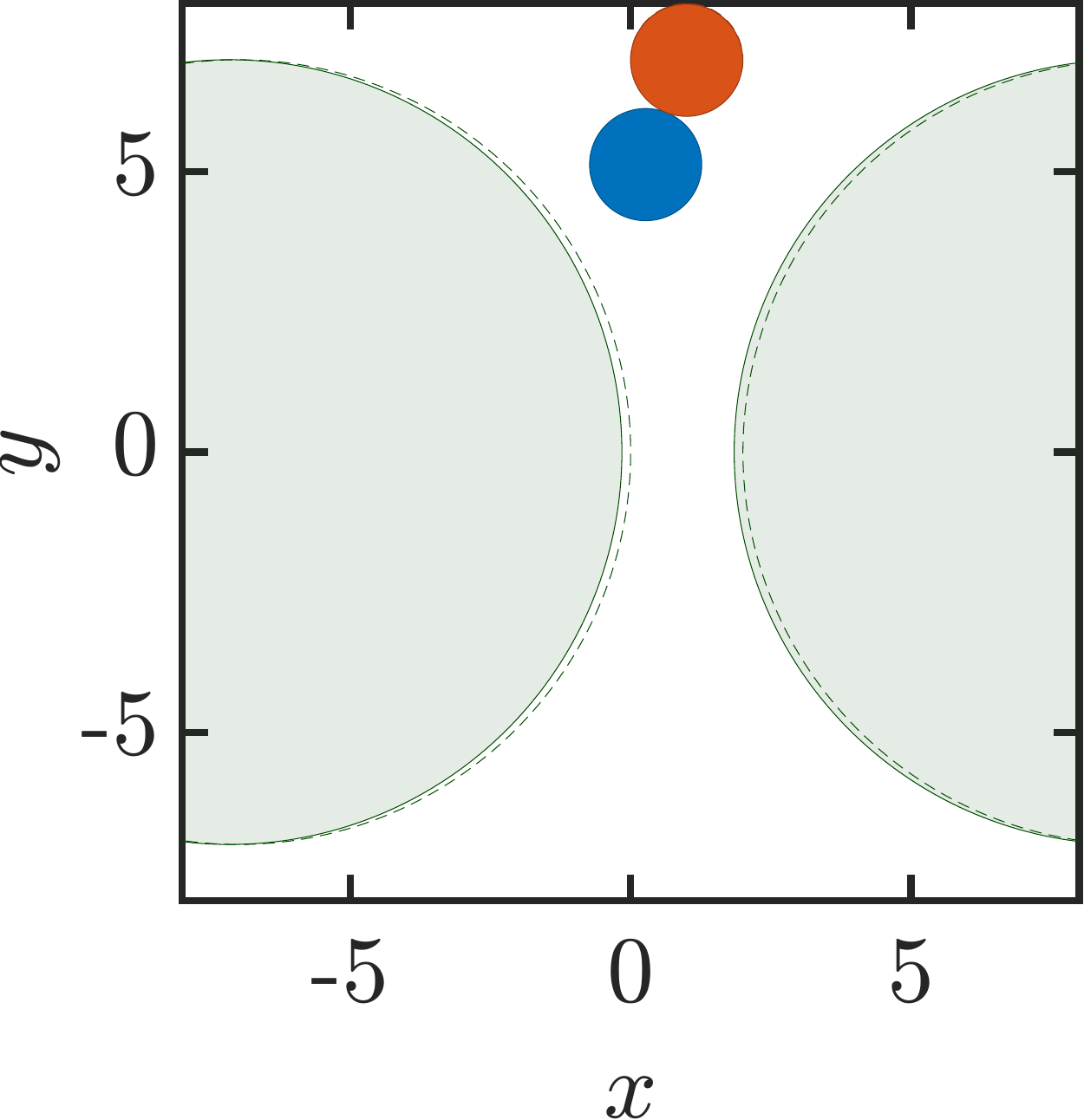}
    \caption{$t=12.0000$}
  \end{subfigure}
  \vspace{-0.2cm}
  \caption{Several frames of the solution for the manipulation problem with autonomous discs.}
  \vspace{-0.8cm}
  \label{fig:auton-discs}
\end{figure*}
In this OCP, we demonstrate a manipulation example with a moving doorway made up of two discs whose centers follow a sinusoidal pattern.
In this problem the goal is for an actuated pusher to push a disc through the moving doorway.
The continuous time OCP is written as:
\begin{mini*}
  {\substack{x(\cdot), u(\cdot)}}
  {\int_{0}^{T}\!u(t)^2 \dd t + 1000\norm{(x_2,x_3) - (1,7)}^2}
  {}
  {}
  \addConstraint{x(t)}{\in (0,0,u(t)) - \normcone{\C}{x},\ }{t\in[0,T]}
  \addConstraint{x(0)}{=\hat{x}}
  \addConstraint{-5}{\le u_i(t)\le 5,}
\end{mini*}
with $x\in\R^4$, $u\in\R^2$, $\hat{x} = (-1,-6,1,-5)$, and time-varying set ${\C(t) = \Set*{z\in\R^4}{c(z,t)\ge 0}}$ with:
\begin{align*}
  c_1(x,t) &= \norm{(x_1,x_2) - (x_3,x_4)}_2^2 - 2^2,\\
  c_2(x,t) &= \norm{(x_1,x_2) - (\cos(t)+8,0)}_2^2 - 8^2,\\
  c_3(x,t) &= \norm{(x_1,x_2) - (\cos(t)-8,0)}_2^2 - 8^2,\\
  c_4(x,t) &= \norm{(x_3,x_4) - (\cos(t)+8,0)}_2^2 - 8^2,\\
  c_5(x,t) &= \norm{(x_3,x_4) - (\cos(t)-8,0)}_2^2 - 8^2,
\end{align*}
which represent contacts between the unactuated slider, the actuated pusher, and the autonomous discs.
These functions form a moving aperture the size of the pusher and slider that follows a sinusoidal motion in the horizontal axis.
Several frames of the solution to this problem are given in \Cref{fig:auton-discs}.
Full solutions and several further examples are available at \href{https://youtu.be/hD1VMQGFEfw}{https://youtu.be/hD1VMQGFEfw}.

\section{Conclusions and Future Work}
In this paper, we show that under certain conditions the solutions to ePDS correspond to solutions of a dynamic complementarity system.
We then show that Moreau's first-order sweeping processes with finitely defined moving sets can be re-formulated as an ePDS with a clock state whose dynamics are not projected.
These reformulations allow us to use a lightly modified form of the finite elements with switch detection discretization to accurately discretize Moreau's first order sweeping processes.
The approaches discussed are implemented in the open source software package \nosnoc.
\bibliographystyle{ieeetran}
\bibliography{syscop}
\end{document}